\documentclass[12pt]{article}
\usepackage{amsmath,amssymb,epsfig,amsthm,bm}
\usepackage{a4wide,xcolor,eucal,enumerate,mathrsfs,graphics,caption,subfig,booktabs,verbatim}
\usepackage{dsfont,natbib}

\theoremstyle{definition}
\newtheorem{de}{Definition}

\theoremstyle{plain}
\newtheorem{theo}[de]{Theorem}
\newtheorem{lemma}[de]{Lemma}

\theoremstyle{remark}

\newcommand{\R}{\mathbb{R}}
\newcommand{\Z}{\mathbb{Z}}
\newcommand{\p}{\mathbb{P}}

\newcommand{\E}{\mathbb{E}}

\newcommand{\F}{\mathcal{F}}

\newcommand{\G}{\mathcal{G}}

\newcommand{\1}{\mathds{1}}

\newcommand{\dd}{\mathrm{d}}

\begin{document}
	\setlength\parindent{0pt}
\centerline{{\Large Results on standard estimators in the Cox model}}

\vskip 5mm
\noindent C\'ecile Durot$^*$ and Eni Musta$^{**}$

\noindent $^*$Universit\'e Paris Nanterre, $^{**}$Katholieke Universiteit Leuven

\noindent $^*$cecile.durot@gmail.com, $^{**}$eni.musta@kuleuven.be

\vskip 3mm

\noindent ABSTRACT

We consider the Cox regression model and prove some properties of the maximum partial likelihood estimator $\hat\beta_n$ and of the 
the Breslow estimator $\Lambda_n$. The asymptotic properties of these estimators have been widely studied in the literature but  we are not aware of a reference where it is shown that they have uniformly bounded moments. These results are needed, for example, when studying global errors of shape restricted estimators of the baseline hazard function. 
\vskip 4mm

\noindent 1.   INTRODUCTION

	We consider the Cox proportional hazards model, which is commonly used to investigate the relationship between the survival times and the predictor variables  in the presence of right censoring. Let $X$ be the event time and $C$ the censoring time for a subject with covariate vector $Z$. We have $n$ i.i.d observations $(T_1,\Delta_1,Z_1),\dots,(T_n,\Delta_n,Z_n)$, where 
$T_i=\min(X_i,C_i)$ is the follow up time and  $\Delta_i=\1_{\{X_i\leq C_i\}}$ is censoring indicator.

The Cox regression model assumes that 
the hazard function at time $t$ for a subject with covariate vector $z\in\R^d$ has the form
\[
\lambda(t|z)=\lambda_0(t)\,\mathrm{e}^{\beta'_0z},\quad t\in\R^+,
\]
where $\lambda_0$ represents the baseline hazard function, corresponding to a subject with $z=0$, and $\beta_0\in\R^d$ is the vector of the regression coefficients. 

The following assumptions are common when studying asymptotics  in the Cox regression model (see  for example Tsiatis (1981), Lopuha\"a and Nane (2013a)). The variable $Z$ has distribution $F_Z(z)$. Given the covariate vector $Z,$ the event time $X$ and the censoring time $C$ are assumed to be independent. Furthermore, conditionally on $Z=z,$ the event time is a nonnegative r.v. with an absolutely continuous distribution function $F(x|z)$ and density $f(x|z).$ The censoring time is a nonnegative r.v. with distribution function $G(x|z)$ and the censoring mechanism is assumed to be non-informative, i.e. $F$ and $G$ share no parameters. We will also need the following assumptions:
\begin{itemize}
	\item[(A1)]
	the end points {$\tau_F$ and $\tau_G$} of the support of $F$ and $G$ satisfy
	\[
	\tau_G<\tau_F\leq\infty\qquad\text{ and }\qquad\p(T=\tau_G)>0,
	\]
	\item[(A2)]
	there exists $\epsilon>0$ such that
	\[
	\sup_{|\beta-\beta_0|\leq\epsilon}\E\left[|Z|^2\,\mathrm{e}^{2\beta'Z}\right]<\infty,
	\]
	\item[(A3)]
	for all $q\geq 1$, we have  $$\E\left[\mathrm{e}^{q\beta'_0Z}\right]<\infty, $$
	\item[(A4)] there exists $\epsilon>0$ such that,
	for all $q\geq 1$ and  $k=1,\dots,d$, we have  $$\sup_{|\beta-\beta_0|\leq\epsilon}\E\left[Z_k^{2q}\mathrm{e}^{q\beta'Z}\right]<\infty. $$
\end{itemize}

Here $|\ . \ |$ denotes  the Euclidean norm, $\beta'$ denotes the transpose of $\beta$ and $Z_k$ is the $k^{th}$ component of the vector $Z$. We will use the index $k=1,\dots,d$ when it corresponds to a component of a vector and indices $i,j=1,\dots,n$ when it corresponds to the different observations. 
The first assumption tells us that, at the end of the study, there is a positive chance that an individual will have survived without being censored. {This assumption is reasonable} in practice since most of the survival studies end at some prespecified time $T_0$ and $C_i=\min(\tilde{C_i},T_0)$, where $\tilde{C_i}$ denotes the censoring time for reasons not related to the end of the study and has support on $[0,\tau_{\tilde{C}}]$ such that $\tau_{\tilde{C}}\geq T_0$. In such a case $\tau_G=T_0$ and $\p(T=T_0)=\p(X\geq T_0,\tilde{C}\geq T_0)>0$ because at the end of the study there are usually subjects that have not experienced the event yet and were still being followed. All those subjects will be censored at time $T_0$.  This assumption is common also in asymptotic studies of the Cox model. In Tsiatis (1981),  it is assumed that the study is terminated at time $T_0$ and  $\p(T\geq T_0)>0$ which means that $\p(T= T_0)>0$. In Andersen and Gill (1982), observations in $[0,1]$ are considered and they assume that their function  $s^{(0)}$ is bounded away from zero on $[0,1]$ (see their condition D on page 1105). This happens only if $\p(T\geq 1)>0$ and in our case $1$ corresponds to $T_0$. 

Assumption  (A2) can be seen as conditions on the boundedness of the second moment of the covariates, for $\beta$ in a neighbourhood of $\beta_0$. The other two assumptions are additional ones needed for our analysis in order to get bounded moments of the maximum partial likelihood estimator $\hat\beta_n$ and of the Breslow estimator $\Lambda_n$ (see  definitions below). 

The proportional hazard property of the Cox model allows estimation of the effects $\beta_0$ of the covariates by the maximum partial likelihood estimator $\hat{\beta}_n$, while leaving the baseline hazard completely unspecified.  
$\hat{\beta}_n$ is defined as  the maximizer of the partial likelihood function
\[
L(\beta)
=
\prod_{i=1}^m 
\frac{\mathrm{e}^{\beta'Z_i}}{\sum_{j=1}^n
	\1_{\{T_j\geq X_{(i)}\}}\mathrm{e}^{\beta'Z_j}},
\]
where  $0<X_{(1)}<\cdots<X_{(m)}<\infty$ denote the ordered, observed event times  (see~ Cox (1972) and Cox (1975)). 
Asymptotic properties of this estimator have been investigated, among other papers,  in Tsiatis (1981), Andersen and Gill (1982). In particular, they show that 
\[
n^{1/2}(\hat\beta_n-\beta_0)\xrightarrow{d}N(0,\Sigma)
\]
for some positive definite matrix $\Sigma$. 
On the other hand,	the nonparametric cumulative baseline hazard
\[
\Lambda_0(t)=\int_0^t\lambda_0(u)\,\mathrm{d}u,
\]
is usually estimated by the Breslow estimator 
\begin{equation}
\label{eq:Breslow}
\Lambda_n(t)=\int \frac{\delta\1_{\{ u\leq t\}}}{\Phi_n(u;\hat{\beta}_n)}\,\mathrm{d}\p_n(u,\delta,z).
\end{equation}
where 
\begin{equation}
\label{eq:def Phin}
\Phi_n(t;\beta)=\int \1_{\{u\geq t\}} \mathrm{e}^{\beta'z}\,\mathrm{d}\p_n(u,\delta,z),
\end{equation}
and $\p_n$ is the empirical measure of the triplets $(T_i,\Delta_i,Z_i)$ with $i=1,\dots,n.$ 
$\Phi_n$ is an estimator of 
\begin{equation}
\label{eq:def Phi}
\Phi(t;\beta)=\int \1_{ \{u\geq t\}}\,\mathrm{e}^{\beta'z}\,\mathrm{d}\p(u,\delta,z),
\end{equation}
{where $\p$ is the common distribution of the triplets $(T_i,\Delta_i,Z_i)$}
and, in Lemma 4 of Lopuha\"a and Nane (2013a) it is shown that
\begin{equation}
\label{eqn:Phi}
\sup_{t\in\R}|\Phi_n(t;\beta_0)-\Phi(t;\beta_0)|=O_p(n^{-1/2}).
\end{equation} 
In the next section, we show that $n^{1/2}\vert\hat\beta_n-\beta_0\vert$,  $n^{1/2}\sup_t\vert\Phi_n(t;\beta_0)-\Phi(t;\beta_0)\vert$ and $n^{1/2}\sup_t\vert\Lambda _n(t)-\Lambda_0(t)\vert$ have uniformly bounded moments of any order. 
Such results are needed, for example, when studying global errors of  isotonic estimators of the baseline hazard (see Appendix D in Durot and Musta (2019)). That paper proposes a goodness of fit test for a parametric baseline distribution based on the $L_p$-distance between an isotonic estimator and a parametric estimator for the baseline hazard. Finding the limit distribution of such test statistic requires studying the asymptotic behaviour of these global errors.  The estimator depends on $\hat\beta_n$ and $\Phi_n$ and since the $L_p$-errors are considered, in the proof it is  needed that both $\hat\beta_n$ and $\Phi_n$ have bounded moments of any order. These results would be needed also for asymptotic study of the supremum distance of estimators of the baseline hazard, which  can be used to construct uniform confidence bands.

\vskip 3mm

\noindent 2. MAIN RESULTS

{We consider first the maximum partial likelihood estimator of the finite dimensional parameter.}

\begin{theo}\label{theo: MLE}
	\label{theo:beta}
	Assume that (A1), (A2), (A4) hold and that the baseline hazard $\lambda_0$ is continuous.
	Let $p\geq 1$. There exist an event $E_n$
	with $\p(E_n)\to 1$, and  $K>0$ such that 
	\begin{equation*}\label{eq: moments betan}
	\limsup_{n\to\infty}\E\left[\1_{E_n}n^{p/2}\vert\hat\beta_n-\beta_0\vert^p\right] \leq K.
	\end{equation*}
\end{theo}

The event $E_n$ in Theorem \ref{theo:beta} is an event with probability converging to one defined in an appropriate way in order to allow us to move from a result on convergence in probability for the matrix of second derivatives, {see \eqref{En} in the proofs,} to a result in terms of expectations. So  {$\hat\beta_n$} has bounded moments  only on such event. However, since its probability converges to one, the result is not restrictive. The same kind of reasoning is used also in other asymptotic studies of the Cox model, see Lopuha\"a and Musta (2018), where a similar event $E_n$ is used to prove the boundedness of the second moments of the isotonic estimators of the baseline hazard. These results are only intermediate results used to find the limiting distribution of certain estimators. Hence for the final convergence in distribution, restricting to an event $E_n$ with probability converging to one is not an issue. 

{Next, we consider the Breslow estimator. Again, the bound is obtained on an event with probability converging to one as $n\to\infty$.}

\begin{theo}\label{theo: breslow}
	Assume that (A1)-(A4) hold and that the baseline hazard $\lambda_0$ is continuous.
	Let $p\geq 1$. There exist an event $A_n$
	with $\p(A_n)\to 1$, and  $K>0$ such that 
	\begin{equation*}
	\label{eq: moments Lambdan}
	\limsup_{n\to\infty}\E\left[\1_{A_n}n^{p/2}\sup_{t\in[0,\tau_G)}\vert\Lambda_n(t)-\Lambda_0(t)\vert^p\right] \leq K.
	\end{equation*}
\end{theo}

{We end up the section with a lemma that is used in the proof of Theorem \ref{theo: breslow} and that is of independent interest. Indeed, the result is used in Durot and Musta (2019) to study  global errors of  isotonic estimators of the baseline hazard.}
\begin{lemma}
	\label{theo:moments_Phin}
	Suppose that (A3) holds.
	Let $p\geq 1$. Then, there exists $K>0$ such that 
	\[
	\sup_{n\geq 1}\E\left[n^{p/2}\sup_{t\in\R}\vert\Phi_n(t;\beta_0)-\Phi(t;\beta_0)\vert^p\right] \leq K.
	\]
\end{lemma}

\vskip 3mm

\noindent 3.  PROOFS

\begin{proof}[Proof of Lemma \ref{theo:moments_Phin}]
	By definition we have 
	\[
	n^{1/2}\sup_{t\in\R}\vert\Phi_n(t;\beta_0)-\Phi(t;\beta_0)\vert=\sup_{t\in\R}\left|\int \1_{ \{u\geq t\}}\,\mathrm{e}^{\beta'_0z}\,\mathrm{d}\sqrt{n}(\p_n-\p)(u,\delta,z)\right|.
	\]
	Let $\F$ be  the class of functions 
	\[
	f_t(u,z)=\1_{ \{u\geq t\}}\,\mathrm{e}^{\beta'_0z},\qquad t\in\R,
	\]
	with envelope function $F(u,z)=\mathrm{e}^{\beta'_0z}$.
	Then, we can write
	\[
	\E\left[n^{p/2}\sup_{t\in\R}\vert\Phi_n(t;\beta_0)-\Phi(t;\beta_0)\vert^p\right]\leq \E\left[\sup_{f\in\F}\left|\int f(u,z) \mathrm{d}\sqrt{n}(\p_n-\p)(u,\delta,z)\right|^p\right]
	\]
	From Theorem 2.14.1  in van der Vaart and Wellner (1996), it follows that 
	\[
	\E\left[\sup_{f\in\F}\left|\int f(u,z) \mathrm{d}\sqrt{n}(\p_n-\p)(u,\delta,z)\right|^p\right]^{1/p}\lesssim J(1,\F)\Vert F\Vert_{L_{2\vee p}(\p)},
	\]
	where
	\[
	J(1,\F)=\sup_{Q}\int_0^1\sqrt{1+\log N(\epsilon\Vert F\Vert_{L_2(Q)},\F,L_2(Q))}\,\dd\epsilon
	\]
	and the supremum is taken over all probability measures $Q$ such that $\Vert F\Vert_{L_2(Q)}>0$. 
	By Assumption (A3), $\Vert F\Vert_{L_{2\vee p}(\p)}<\infty$. Hence, it remains to show that  $J(1,\F)$ is bounded. 
	
	Let $Q$ be a probability measure on $\R\times\R^d$ such that $\Vert F\Vert_{L_2(Q)}>0$. Let $Q'$ be the probability measure on $\R$ defined by 
	\[
	Q'(S)=\frac{\int_{S\times\R^p}e^{2\beta'_0z}\,\dd Q(u,z)}{\int_{\R\times \R^p}e^{2\beta'_0z}\,\dd Q(u,z)}=\frac{\int_{S\times\R^p}e^{2\beta'_0z}\,\dd Q(u,z)}{\Vert F\Vert_{L_2(Q)}^2},\qquad S\subseteq \R.
	\]
	For a given $\epsilon>0$ select an $\epsilon$-net $g_1,\dots,g_N$ in the class $\G$ of monotone functions $g:\,\R\to[0,1]$ with respect to $L_{2}(Q')$. From Theorem 2.7.5  and in van der Vaart and Wellner (1996) and the relation between covering and bracketing numbers on page 84 of van der Vaart and Wellner (1996), we have 
	$
	N\lesssim 1/\epsilon
	$ and the constant in the inequality $\lesssim$ does not depend on $Q'$.  Next, we consider functions $f_i(u,z)=g_i(u)e^{\beta'_0z}$. Then $f_1,\dots f_N$ form an $\epsilon\Vert F\Vert_{L_2(Q)}$-net of the class $\F$ with respect to $L_2(Q)$. Indeed, for each $t\in\R$, let $i$ be such that $g(u)=\1_{\{u\geq t\}}$ belongs in the $\epsilon$-ball around $g_i$. Then
	\[
	\begin{split}
	\Vert f_t-f_i\Vert_{L_2(Q)}^2&=\int_{\R\times\R^p}\left(\1_{\{u\geq t\}}-{g_i}(u)\right)^2e^{2\beta'_0z}\,\dd Q(u,z)\\
	&=\Vert F\Vert_{L_2(Q)}^2\int_{\R}\left(\1_{\{u\geq t\}}-{g_i}(u)\right)^2\,\dd Q'(u)\\
	&=\Vert F\Vert_{L_2(Q)}^2\Vert g-g_i\Vert_{L_2(Q')}^2\\
	&\leq \epsilon^2\Vert F\Vert_{L_2(Q)}^2.
	\end{split}
	\]
	Therefore
	\[
	N(\epsilon\Vert F\Vert_{L_2(Q)},\F,L_2(Q))\leq \frac{K}{\epsilon}
	\]
	for some constant $K>0$ independent of $Q$. It follows that $J(1,\F)$ is bounded, which concludes the proof.
\end{proof}

\begin{proof}[Proof of Theorem \ref{theo: MLE}]
	Let $S(\beta)$ be the log partial likelihood function
	\[
	S(\beta)=\log L(\beta)=\sum_{i=1}^m \beta' Z_{(i)}-\sum_{i=1}^m \log \left(\sum_{j=1}^n\1_{\{T_j\geq X_{(i)}\}}e^{\beta'Z_j} \right)
	\]
	where $X_{(1)},\dots,X_{(m)}$ are the ordered observed event times. From Theorem 3.1 in Tsiatis (1981),  $\hat\beta_n$ is the solution of $S'(\beta)=0$, where $S'$ denotes the vector $\left(\frac{\partial S(\beta)}{\partial \beta_1},\dots,\frac{\partial S(\beta)}{\partial \beta_d}\right)$. Note that, in Tsiatis (1981) it is written that $\hat\beta_n$ is  the solution to the equation (3.2) but actually it is a zero of the expression in (3.2).
	By a Taylor expansion 
	we have
	\[
	S'(\hat{\beta}_n)=S'(\beta_0)-S''(\beta^*)\left(\hat\beta_n-\beta_0\right)=0,
	\]
	where $|\beta^*-\beta_0|\leq|\hat\beta_n-\beta_0|$ and the positive semi-definite matrix $S''$ is minus the matrix of the second derivatives $S''_{ij}{(\beta)}=-\frac{\partial^2 S(\beta)}{\partial \beta_j\partial \beta_i}$ . We also know that $\frac1nS''(\beta^*)$ converges in probability to a nonsingular matrix $\Sigma$, see {the second step of the proof of} Theorem 3.2 in Andersen and Gill (1982).  There $S''$ is denoted by $\mathcal{I}$. In this proof conditions A, B, D of  Andersen and Gill (1982) are used. In our setting A is satisfied because we are assuming a continuous hazard rate. For B note that their $S^{(0)}$, $S^{(1)}$ and $S^{(2)}$ correspond to our $\Phi_n$, 
	\begin{equation}
	\label{eqn:D^1}
	D_n^1(t;\beta)=\frac{\partial \Phi_n(t;\beta)}{\partial \beta}=\frac1n\sum_{i=1}^n \1_{\{T_i\geq t\}}Z_ie^{\beta'_0Z_i}.
	\end{equation}
	and
	\begin{equation}
	\label{eqn:D^2}
	D^2_n(t;\beta)=\frac{\partial^2 \Phi_n(t;\beta)}{\partial\beta^2}=\frac1n\sum_{i=1}^nZ_iZ'_i\1_{\{T_i\geq t\}}e^{\beta'Z_i}.
	\end{equation}
	which converge uniformly to $\Phi$, 
	\[
	D^1(t;\beta)=\frac{\partial \Phi(t;\beta)}{\partial \beta},\quad\text{ and }\quad D^2(t;\beta)=\frac{\partial^2 \Phi(t;\beta)}{\partial\beta^2}
	\]
	(See Lemma 3.1 in Lopuha\"a and Nane (2013b) for the first two; in the same way one can also deal with $D^2_n$). For condition D in Andersen and Gill (1982), the boundedness of $D^1$ and $D^2$ follows from our assumptions (A2) and (A4). They  consider observations in a compact  interval  $[0,1]$ such that $\p(T\geq 1)>0$, in order to have $\inf_{t\in[0,1]}\Phi(t)>0$.  Here we consider {observations} on $[0,\tau_G]$ with $\p(T=\tau_G)>0$, so $\inf_{t\in[0,\tau_G]}\Phi(t;\beta_0)=\Phi(\tau_G;\beta_0)>0$. If assumption (A1) was not satisfied, i.e. $\p(T=\tau_G)=0$, we would need to restrict our results on an interval $[0,M]$ with $M<\tau_G$ in order to have $\inf_{t\in[0,M]}\Phi(t;\beta_0)>0$.
	Hence 
	\[
	\sqrt{n}\left(\hat\beta_n-\beta_0\right)=\Sigma^{-1}n^{-1/2}S'(\beta_0)-\left(\Sigma^{-1}\frac1nS''(\beta^*)-I\right)\sqrt{n}\left(\hat\beta_n-\beta_0\right).
	\]
	It follows that 
	\[
	\begin{split}
	\sqrt{n}\left|\hat\beta_n-\beta_0\right|&\leq \left|\Sigma^{-1}n^{-1/2}S'(\beta_0)\right|+\left|\left(n^{-1}\Sigma^{-1}S''(\beta^*)-I\right)\sqrt{n}\left(\hat\beta_n-\beta_0\right) \right|\\
	&\leq \Vert \Sigma^{-1}\Vert\, |n^{-1/2}S'(\beta_0)|+\Vert n^{-1}\Sigma^{-1}S''(\beta^*)-I\Vert \, \sqrt{n}\left|\hat\beta_n-\beta_0\right|
	\end{split}
	\]
	where $|\cdot|$ is the Euclidean norm in $\R^d$ and $\Vert\cdot\Vert$ is the matrix norm induced by the Euclidean vector norm, i.e. 
	\[
	\Vert A\Vert =\sup_{x\in\R^d\setminus\{0\}}\frac{|Ax|}{|x|}=\sigma_{\max}(A)\leq \left(\sum_{i,j=1}^dA_{ij}^2\right),\qquad A\in \R^{d\times d}
	\]
	and $\sigma_{\max}(A)$ is the largest singular value of $A$.
	Let $\epsilon<1$. Since $n^{-1}S''(\beta^*)\to \Sigma$ in probability, we can take the event 
	\[
	E_{n}=\left\{\Vert n^{-1}\Sigma^{-1}S''(\beta^*)-I\Vert\leq \epsilon\right\}.
	\]
	Then, we have $\p(E_{n})\to 1$ and
	\begin{equation}\label{En}
	\1_{E_{n}}\sqrt{n}\left|\hat\beta_n-\beta_0\right|\leq\frac{1}{1-\epsilon}  \Vert \Sigma^{-1}\Vert\, |n^{-1/2}S'(\beta_0)|.
	\end{equation}
	It suffices to show that $\E\left[|n^{-1/2}S'(\beta_0)|^p\right]$ is uniformly bounded.

	By definition we have 
	\[
	S'(\beta_0)=\sum_{i=1}^{n}\Delta_iZ_i-\sum_{i=1}^n \Delta_i\frac{D_n^1(T_i;\beta_0)}{\Phi_n(T_i;\beta_0)}
	\]
	where $D_n^1(t;\beta)$ is defined as in \eqref{eqn:D^1}.
	We will follow the martingale approach of Kalbfleisch and Prentice (2002). For each $i=1,\dots,n$, let $N_i(t)=\Delta_i\1_{\{T_i\leq t\}}$ be the right-continuous counting process for the number of observed failures on $(0,t]$ and $Y_i(t)=\1_{\{T_i\geq t\}}$ be the at-risk process.
	{From (5.49) in Kalbfleisch and Prentice (2002), the compensator of $N_i(t)$ is}
	\[
	A_i(t)=\int_0^t Y_i(u)\lambda_0(u)e^{\beta'_0Z_i}\,\dd u,
	\]
	{and 
		$M_i(t)=N_i(t)-A_i(t)$}
	is a mean zero martingale with respect to the filtration $$\F_t=\{N_i(s),Y_i(s+),Z_i\,:\,i=1,\dots,n, \, s\in[0,t]\}$$ (see Kalbfleisch and Prentice (2002), page 173). 
	The score function $S'$ up to a certain time $t$ can be then written 
	\[
	\begin{split}
	S'(\beta_0,t)&=\sum_{i=1}^n\int_0^t\left[Z_i-\frac{D_n^1(u;\beta_0)}{\Phi_n(u;\beta_0)}\right]\,\dd N_i(u)\\
	&=\sum_{i=1}^n\int_0^t\left[Z_i-\frac{D_n^1(u;\beta_0)}{\Phi_n(u;\beta_0)}\right]\,\dd M_i(u)
	\end{split}
	\]
	(see equations (5.50) and (5.51) in Kalbfleisch and Prentice (2002)). {Note that we can replace $\dd N_i$ by $\dd M_i$ because} 
	\[
	\begin{split}
	&\sum_{i=1}^n\int_0^t\left[Z_i-\frac{D_n^1(u;\beta_0)}{\Phi_n(u;\beta_0)}\right]\,\dd A_i(u)\\
	&=\sum_{i=1}^n\int_0^t\left[Z_i-\frac{D_n^1(u;\beta_0)}{\Phi_n(u;\beta_0)}\right]Y_i(u)\lambda_0(u)e^{\beta'_0Z_i}\,\dd u\\
	&=\int_0^t\lambda_0(u)\sum_{i=1}^nZ_iY_i(u)e^{\beta'_0Z_i}\,\dd u-\int_0^t\frac{D_n^1(u;\beta_0)}{\Phi_n(u;\beta_0)}\lambda_0(u)\sum_{i=1}^nY_i(u)e^{\beta'_0Z_i}\,\dd u\\
	&=n\int_0^t\lambda_0(u)D^1_n(u;\beta_0)\,\dd u-n\int_0^t D_n^1(u;\beta_0)\lambda_0(u)\,\dd u=0.
	\end{split}
	\]
	Being a sum of stochastic integrals of predictable processes with respect to a martingale, $S'{(\beta_0,\ .\ )}$ is also an $\F_t$-martingale. Let $$G_{i,n}(u)=\left[Z_i-\frac{D_n^1(u;\beta_0)}{\Phi_n(u;\beta_0)}\right].$$Then
	\[
	n^{-1/2}S'(\beta_0,t)=n^{-1/2}\sum_{i=1}^n\int_0^tG_{i,n}(u)\,\dd M_i(u)
	\]
	is a martingale with predictable variation process
	\[
	\langle n^{-1/2}S'(\beta_0)\rangle_t=\int_0^t \frac1n\sum_{i=1}^n \left\{G_{i,n}(u)G_{i,n}'(u)Y_i(u)e^{\beta'_0Z_i} \right\}\lambda_0(u)\,\dd u
	\]
	(see proof of (5.58) in page 176 of Kalbfleisch and Prentice (2002)). We have 
	\[
	\begin{split}
	&\E\left[\left|n^{-1/2}S'(\beta_0,t)\right|^p\right]\\
	&=\E\left[n^{-p/2}\left(\sum_{i=1}^n\int_0^tG_{i,n}'(u)\,\dd M_i(u)\sum_{i=1}^n\int_0^tG_{i,n}(u)\,\dd M_i(u)\right)^{p/2}\right]\\
	&=\E\left[n^{-p/2}\left(\sum_{k=1}^d\left( \sum_{i=1}^n\int_0^t\left(G_{i,n}(u)\right)_k\,\dd M_i(u)\right)^2\right)^{p/2}\right]\\
	&\lesssim \sum_{k=1}^d\E\left[n^{-p/2}\left( \sum_{i=1}^n\int_0^t\left(G_{i,n}(u)\right)_k\,\dd M_i(u)\right)^{p}\right],
	\end{split}
	\]
	where again $G_{i,n}'(u)$ denotes the transpose of the vector $G_{i,n}(u)$ and $(G_{i,n}(u))_k$ denotes its $k^{th}$ component. For the first and the second equalities we have used the definition of the Euclidean norm of a vector in $\R^d$, while for the last inequality we use that for positive numbers $a_1,\dots,a_d$   and all $p$ we have $(a_1+\dots+a_d)^p\leq d^p(a_1^p+\dots a_d^p)$.  Each component 
	\[
	\sum_{i=1}^n\int_0^t\left(G_{i,n}(u)\right)_k\,\dd M_i(u)
	\]
	is a martingale with quadratic variation
	\[
	\left\langle \sum_{i=1}^n\int_0^t\left(G_{i,n}(u)\right)_k\,\dd M_i(u)\right\rangle=\sum_{i=1}^n\int_0^t\left(G_{i,n}(u)\right)_k^2Y_i(u)e^{\beta'_0Z_i}\lambda_0(u)\,\dd u.
	\]
	It follows from properties of stochastic integrals  that
	\[
	\begin{split}
	&\E\left[\left|n^{-1/2}S'(\beta_0,t)\right|^p\right]\\
	&\lesssim \sum_{k=1}^d\E\left[n^{-p/2}\left\langle \sum_{i=1}^n\int_0^t\left(G_{i,n}(u)\right)_k\,\dd M_i(u)\right\rangle^{p/2}\right]\\
	&\lesssim \sum_{k=1}^d\E\left[\left(\frac1n \sum_{i=1}^n\int_0^t\left(G_{i,n}(u)\right)_k^2Y_i(u)e^{\beta'_0Z_i}\lambda_0(u)\,\dd u\right)^{p/2}\right].
	\end{split}
	\]
	Note that 
	\[
	\begin{split}
	&\frac1n \sum_{i=1}^n\int_0^t\left(G_{i,n}(u)\right)_k^2Y_i(u)e^{\beta'_0Z_i}\lambda_0(u)\,\dd u\\
	&\lesssim \int_0^t \frac1n\sum_{i=1}^n(Z_i)_k^2Y_i(u)e^{\beta'_0Z_i}\,\dd u+\int_0^t\frac{(D^1_n(u;\beta_0))_k^2}{\Phi_n(u;\beta_0)^2}\frac1n\sum_{i=1}^nY_i(u)e^{\beta'_0Z_i}\,\dd u\\
	&\lesssim \int_0^t (D^2_n(u;\beta_0))_{kk}\,\dd u+\int_0^t\frac{(D^1_n(u;\beta_0))_k^2}{\Phi_n(u;\beta_0)}\,\dd u\\
	&\lesssim \sup_{u\in[0,t]}(D^2_n(u;\beta_0))_{kk}+\sup_{u\in[0,t]}\frac{(D^1_n(u;\beta_0))_k^2}{\Phi_n(u;\beta_0)},
	\end{split}
	\]
	where $D^2_n(u;\beta)$ is defined in \eqref{eqn:D^2}.
	Note that $S'(\beta_0)$ is equal to $S'(\beta_0,T_{(n)})$. Hence, in order to have $\E\left[\left|n^{-1/2}S'(\beta_0)\right|^p\right]$ uniformly bounded, it suffices to show that, for all $p\geq 1$, 
	\begin{equation}
	\label{eqn:expectations}
	\E\left[\sup_{u\in[0,T_{(n)}]}(D^2_n(u;\beta_0))_{kk}^p\right]\qquad \mbox{and}\qquad \E\left[\sup_{u\in[0,T_{(n)}]}\frac{(D^1_n(u;\beta_0))_k^{2p}}{\Phi_n(u;\beta_0)^p}\right]
	\end{equation}
	are uniformly bounded, where $T_{(n)}$ is the largest of the observations $T_1,\dots, T_n$. 
	
	By definition, we have
	\[
	\sup_{u\in[0,T_{(n)}]}(D^2_n(u;\beta_0))_{kk}\leq \frac{1}{n}\sum_{i=1}^n (Z_i)_{k}^2e^{\beta'_0Z_i}
	\]
	Also $1/\Phi_n$ is well defined up to $T_{(n)}$
	and, from Titu's lemma, 
	\begin{equation}
	\label{eqn:Titu}
	\sup_{u\in[0,T_{(n)}]}\frac{(D^1_n(u;\beta_0))_k^2}{\Phi_n(u;\beta_0)}\leq \frac{1}{n}\sum_{i=1}^n (Z_i)_{k}^2e^{\beta'_0Z_i} 
	\end{equation}
	Hence, {in order to show that the expectations in \eqref{eqn:expectations} are bounded, it suffices to show that 
		\[\E\left[\left(\frac{1}{n}\sum_{i=1}^n (Z_i)_{k}^2e^{\beta'_0Z_i}\right)^p\right]\]
		is bounded.}
	Let $J=\{a=(a_1,\dots,a_n)\in\Z^n,\,a_i\geq 0 \text{ for all }i=1,\dots,n, \, \sum_{i=1}^n a_i=p \}$. Then, using linearity of the expectation, independence of the {$Z_i$'s}, it follows that 
	\[
	\begin{split}
	&\E\left[\left( \frac{1}{n}\sum_{i=1}^n (Z_i)_{k}^2e^{\beta'_0Z_i} \right)^p\right]\\
	&=\frac{1}{n^p}\sum_{a\in J}\binom{p}{a_1,\dots,a_n}\E\left[\prod_{i=1}^n(Z_i)_{k}^{2a_i}e^{a_i\beta'_0Z_i}\right]\\
	&=\frac{1}{n^p}\sum_{a\in J}\binom{p}{a_1,\dots,a_n}\prod_{i=1}^n\E\left[Z_{k}^{2a_i}e^{a_i\beta'_0Z}\right],
	\end{split}
	\]
	where $\binom{p}{a_1,\dots,a_n}$ are the multinomial coefficients. Using iteratively that, for a positive random variable $Y$ and $a,b\geq 0$, we have {$\E[Y^{a+b}]-\E[Y^a]\E[Y^b]=Cov(Y^a,Y^b)\geq 0$}, we obtain
	\[
	\prod_{i=1}^n\E\left[Z_{k}^{2a_i}e^{a_i\beta'_0Z}\right]\leq \E\left[Z_{k}^{2\sum_{i=1}^na_i}e^{\sum_{i=1}^na_i\beta'_0Z}\right]=\E\left[Z_{k}^{2p}e^{p\beta'_0Z}\right]. 
	\]
	Therefore, since $\sum_{a\in J}\binom{p}{a_1,\dots,a_n}=n^p$, we have
	\[
	\E\left[\left( \frac{1}{n}\sum_{i=1}^n (Z_i)_{k}^2e^{\beta'_0Z_i} \right)^p\right]\leq \E\left[Z_{k}^{2p}e^{p\beta'_0Z}\right] \frac{1}{n^p}\sum_{a\in J}\binom{p}{a_1,\dots,a_n}=\E\left[Z_{k}^{2p}e^{p\beta'_0Z}\right].
	\]
	By assumption (A4) it follows that $\E\left[\left( \frac{1}{n}\sum_{i=1}^n (Z_i)_{k}^2e^{\beta'_0Z_i} \right)^p\right]$, and as a result also $\E\left[\left|n^{-1/2}S'(\beta_0)\right|^p\right] $, are uniformly bounded.
	This concludes the proof of the theorem.
\end{proof}

\begin{proof}[Proof of Theorem \ref{theo: breslow}]
	For $t\in[0,\tau_G) $ we can express the cumulative baseline hazard as 
	\[
	\Lambda_0(t)=\int \frac{\delta\1_{\{u\leq t \}}}{\Phi(u;\beta_0)}\,\dd \p(u,\delta,z)
	\]
	(see (10) in Lopuha\"a and Nane (2013a)). 
	Hence, by definition of $\Lambda_n$ and the triangular inequality,  we have
	\begin{equation}
	\label{eqn:decomposition}
	\begin{split}
	\vert\Lambda_n(t)-\Lambda_0(t)\vert&\leq \left|\int \frac{\delta\1_{\{u\leq t \}}}{\Phi(u;\beta_0)}\,\dd (\p_n-\p)(u,\delta,z) \right|\\
	&\quad+\left|\int \delta\1_{\{u\leq t \}}\left(\frac{1}{\Phi_n(u;\hat\beta_n)}-\frac{1}{\Phi_n(u;\beta_0)}\right)\,\dd \p_n(u,\delta,z) \right|\\
	&\quad+\left|\int \delta\1_{\{u\leq t \}}\left(\frac{1}{\Phi_n(u;\beta_0)}-\frac{1}{\Phi(u;\beta_0)}\right)\,\dd \p_n(u,\delta,z) \right|.
	\end{split}
	\end{equation}
	Hence, it suffices to show that there exists an event $A_n$ with $\p(A_n)\to 1$ and positive constants $K_1,$ $K_2$, $K_3$ such that 
	\begin{equation}
	\label{eqn:term1}
	\limsup_{n\to\infty}\E\left[n^{p/2}\sup_{t\in[0,\tau_G)}\left|\int \frac{\delta\1_{\{u\leq t \}}}{\Phi(u;\beta_0)}\,\dd (\p_n-\p)(u,\delta,z) \right|^p\right] \leq K_1,
	\end{equation} 
	\begin{equation}
	\label{eqn:term2}
	\limsup_{n\to\infty}\E\left[\1_{A_n}n^{p/2}\sup_{t\in[0,\tau_G)}\left|\int \delta\1_{\{u\leq t \}}\left(\frac{1}{\Phi_n(u;\beta_0)}-\frac{1}{\Phi(u;\beta_0)}\right)\,\dd \p_n(u,\delta,z) \right|^p\right] \leq K_2
	\end{equation} 
	and
	\begin{equation}
	\label{eqn:term3}
	\limsup_{n\to\infty}\E\left[\1_{A_n}n^{p/2}\sup_{t\in[0,\tau_G)}\left|\int \delta\1_{\{u\leq t \}}\left(\frac{1}{\Phi_n(u;\hat\beta_n)}-\frac{1}{\Phi_n(u;\beta_0)}\right)\,\dd \p_n(u,\delta,z) \right|^p\right]\leq K_3. 
	\end{equation} 
	For \eqref{eqn:term1}, consider the class  $\F$ of functions 
	\[
	f_t(u,\delta,z)=\frac{\delta\1_{\{u\leq t \}}}{\Phi(u;\beta_0)},\qquad t\in[0,\tau_G),
	\]
	with envelope function $F(u,\delta,z)=\frac{\delta\1_{\{u\leq \tau_G \}}}{\Phi(\tau_G;\beta_0)}$.
	Then, we can write
	\[
	\E\left[n^{p/2}\sup_{t\in[0,\tau_G)}\left|\int \frac{\delta\1_{\{u\leq t \}}}{\Phi(u;\beta_0)}\,\dd (\p_n-\p)(u,\delta,z) \right|^p\right]\leq \E\left[\sup_{f\in\F}\left|\int f(u,\delta,z) \mathrm{d}\sqrt{n}(\p_n-\p)(u,\delta,z)\right|^p\right]
	\]
	From Theorem 2.14.1  in van der Vaart and Wellner (1996), it follows that 
	\[
	\E\left[\sup_{f\in\F}\left|\int f(u,\delta,z) \mathrm{d}\sqrt{n}(\p_n-\p)(u,\delta,z)\right|^p\right]^{1/p}\lesssim J(1,\F)\Vert F\Vert_{L_{2\vee p}(\p)},
	\]
	where
	\[
	J(1,\F)=\sup_{Q}\int_0^1\sqrt{1+\log N(\epsilon\Vert F\Vert_{L_2(Q)},\F,L_2(Q))}\,\dd\epsilon
	\]
	and the supremum is taken over all probability measures $Q$ such that $\Vert F\Vert_{L_2(Q)}>0$. 
	By Assumption (A1), $\Phi(\tau_G;\beta_0)>0$, so $\Vert F\Vert_{L_{2\vee p}(\p)}<\infty$. Moreover, $f_t$ is a product of a bounded monotone function with an indicator function. 
	From Theorem 2.7.5  in van der Vaart and Wellner (1996) and the relation between covering and bracketing numbers on page 84 of van der Vaart and Wellner (1996), we have 
	$
	N(\epsilon\Vert F\Vert_{L_2(Q)},\F,L_2(Q))\lesssim 1/\epsilon\Vert F\Vert_{L_2(Q)}
	$ and the constant in the inequality $\lesssim$ does not depend on $Q$.  
	It follows that   $J(1,\F)$ is bounded. This concludes the proof of \eqref{eqn:term1}.

	For \eqref{eqn:term2}, by the mean value theorem, we have
	\[
	\begin{split}
	&\sup_{t\in[0,\tau_G)}\left|\int \delta\1_{\{u\leq t \}}\left(\frac{1}{\Phi_n(u;\beta_0)}-\frac{1}{\Phi(u;\beta_0)}\right)\,\dd \p_n(u,\delta,z) \right|\\
	&\leq \sup_{{u\in[0,\tau_G)}}\left|\frac{1}{\Phi_n(u;\beta_0)}-\frac{1}{\Phi(u;\beta_0)}\right|\\
	&\leq\frac{\sup_{t\in[0,\tau_G)}\left|\Phi_n(t;\beta_0)-\Phi(t;\beta_0)\right|}{\min\{\inf_{t\in[0,\tau_G)}\Phi_n(t;\beta_0)^2,\inf_{t\in[0,\tau_G)}\Phi(t;\beta_0)^2\}} 
	\end{split}.
	\]
	By assumption (A1), $\inf_{t\in[0,\tau_G)}\Phi(t;\beta_0)=\Phi(\tau_G;\beta_0)>0$. Define the event \[
	A^1_n=\left\{\left|\Phi_n(\tau_G;\beta_0)-\Phi(\tau_G;\beta_0)\right|\leq \frac{1}{2}\Phi(\tau_G;\beta_0)\right\}
	\]
	From \eqref{eqn:Phi} it follows that  $\p(A^1_n)\to 1$ as $n\to \infty$ and on $A^1_n$ we have 
	\[
	\frac{1}{\inf_{t\in[0,\tau_G)}\Phi_n(t;\beta_0)^2}=\frac{1}{\Phi_n(\tau_G;\beta_0)^2}\leq \frac{{4}}{\Phi(\tau_G;\beta_0)^2}<\infty.
	\]
	Hence, by Lemma \ref{theo:moments_Phin}, we obtain 
	\[
	\begin{split}
	&\limsup_{n\to\infty}\E\left[\1_{A_n^1}n^{p/2}\sup_{t\in[0,\tau_G)}\left|\int \delta\1_{\{u\leq t \}}\left(\frac{1}{\Phi_n(u;\beta_0)}-\frac{1}{\Phi(u;\beta_0)}\right)\,\dd \p_n(u,\delta,z) \right|^p\right]\\
	&\lesssim \limsup_{n\to\infty}\E\left[n^{p/2}\sup_{t\in[0,\tau_G)}\left|\Phi_n(t;\beta_0)-\Phi(t;\beta_0)\right|^p\right]\leq K_2.
	\end{split}
	\]

	Finally, for \eqref{eqn:term3}, we have
	\[
	\begin{split}
	&\sup_{t\in[0,\tau_G)}\left|\int \delta\1_{\{u\leq t \}}\left(\frac{1}{\Phi_n(u;\hat\beta_n)}-\frac{1}{\Phi_n(u;\beta_0)}\right)\,\dd \p_n(u,\delta,z) \right|\\
	&\leq \sup_{{u\in[0,\tau_G)}}\left|\frac{1}{\Phi_n(u;\hat\beta_n)}-\frac{1}{\Phi_n(u;\beta_0)}\right|\\
	&\leq\left|\hat\beta_n-\beta_0\right| \sup_{t\in[0,\tau_G)}\frac{\left|D_n^1(t;\beta^*)\right|}{\Phi_n(t;\beta^*)^2}
	\end{split}.
	\]
	for some $\beta^*$ such that $|\beta^*-\beta_0|\leq |\hat\beta_n-\beta_0|,$ {where $D_n^1$ is defined as in \eqref{eqn:D^1}.} 
	As in \eqref{eqn:Titu} we can write 
	\[
	\sup_{t\in[0,\tau_G)}\frac{D_n^1(t;\beta^*)}{\Phi_n(t;\beta^*)^2}\leq \left(\frac{1}{n}\sum_{i=1}^n Z_i^2e^{\beta'_0Z_i}\right)^{1/2}\frac{1}{\inf_{t\in[0,\tau_G)}\Phi_n(t;\beta^*)^{3/2}},
	\]
	where the previous inequality holds componentwise.
	{Similar to Lemma 4 in Lopuha\"a and Nane (2013a), it can be shown} that
	\[
	\sup_{t\in\R}\left|\Phi_n(t;\beta^*)-\Phi(t;\beta_0)\right|\to 0,\qquad\text{in probability}.
	\]
	Define the event $A_n^2=\left\{\sup_{t\in\R}\left|\Phi_n(t;\beta^*)-\Phi(t;\beta_0)\right|\leq \frac{1}{2}\Phi(\tau_G;\beta_0)\right\}$. Then, $\p(A^2_n)\to 1$ as $n\to \infty$ and on $A^2_n$ we have 
	\[
	\frac{1}{\inf_{t\in[0,\tau_G)}\Phi_n(t;\beta^*)^{3/2}}=\frac{1}{\Phi_n(\tau_G;\beta^*)^{3/2}}\leq \frac{{2^{3/2}}}{\Phi(\tau_G;\beta_0)^{3/2}}<\infty.
	\]
	Let $A^3_n=A^2_n\cap E_n$, where $E_n$ is the event in Theorem \ref{theo:beta}. We again have $\p({A^3_n})\to 1$ and by Cauchy-Schwartz inequality we obtain 
	\[
	\begin{split}
	&\limsup_{n\to\infty}\E\left[\1_{A_n^3}n^{p/2}\sup_{t\in[0,\tau_G)}\left|\int \delta\1_{\{u\leq t \}}\left(\frac{1}{\Phi_n(u;\hat\beta_n)}-\frac{1}{\Phi_n(u;\beta_0)}\right)\,\dd \p_n(u,\delta,z) \right|^p\right]\\
	&\lesssim\limsup_{n\to\infty}\E\left[\1_{E_n}n^{p/2}\left|\hat\beta_n-\beta_0\right|^p \left|\frac{1}{n}\sum_{i=1}^n Z_i^2e^{\beta'_0Z_i}\right|^{p/2}\right]\\
	&\leq \left(\limsup_{n\to\infty}\E\left[\1_{E_n}n^{p}\left|\hat\beta_n-\beta_0\right|^{2p}\right]\right)^{1/2}\left(\limsup_{n\to\infty}\E\left[ \left|\frac{1}{n}\sum_{i=1}^n {Z_i^2e^{(\beta^*)'Z_i}}\right|^{p}\right]\right)^{1/2}.
	\end{split}
	\]
	The first term on the right hand side is bounded because of Theorem \ref{eq: moments betan} while the second term is shown to be bounded in the proof of Theorem \ref{theo:beta} using assumption $(A4)$.
	To conclude, the statement of the theorem holds if we take $A_n=A^1_n\cap A^3_n$.
\end{proof}
\noindent BIBLIOGRAPHY
\vskip 3mm

\noindent Andersen, P. K. and Gill, R. D. (1982),
	Cox's regression model for counting processes: a large sample
		study. {\it Ann. Statist.}, {\bf 10.4},
	{1100--1120}.
\vskip 3mm
\noindent Cox, D. R. (1972), Regression models and life-tables.
	 {\it J. Roy. Statist. Soc. Ser. B},
	{\bf 34},
	{187--220}.
\vskip 3mm
\noindent Cox, D. R.(1975),
	Partial likelihood.
   {\it Biometrika},
{\bf 62.2},
{269--276}.
\vskip 3mm
\noindent Durot, C.  and Musta, E. (2019),
 On the $L_p$ error of the Grenander-type estimator in the Cox model.
 { https://arxiv.org/abs/1907.06933
}.

\vskip 3mm
\noindent Kalbfleisch, J. D. and Prentice, R. L. (2002),
The statistical analysis of failure time data.
{ \it Wiley Series in Probability and Statistics},
{Second edition},
{Wiley-Interscience [John Wiley \& Sons], Hoboken, NJ},
{xiv+439}.

\vskip 3mm
\noindent 	Lopuha{\"a}, H. P. and Musta, E. (2018), Smoothed isotonic estimators of a monotone baseline hazard in the Cox model.
{\it Scand. J. Stat.},
{\bf 45.3},
{753--791}.

\vskip 3mm
\noindent Lopuha{\"a},  H. P. and Nane,  G. F.   (2013a),
An Asymptotic Linear Representation for the Breslow Estimator.
 {\it Comm. Statist. Theory Methods},
{\bf 42.7},
{1314-1324}.
\vskip 3mm
\noindent  Lopuha{\"a}, H. P. and Nane, G. F. (2013b),
Shape constrained non-parametric estimators of the baseline
	distribution in {C}ox proportional hazards model.
{\it Scand. J. Stat.},
 {\bf 40.3},
{619--646}.
\vskip 3mm
\noindent Tsiatis, A. A. (1981),
A large sample study of {C}ox's regression model.
{\it Ann. Statist.},
{\bf 9.1},
{93--108}.
\vskip 3mm
\noindent van der Vaart, A. W. and Wellner, J. A. (1996),
Weak convergence and empirical processes.
{\it Springer Series in Statistics},
{Springer-Verlag, New York}, {with applications to statistics},
{xvi+508}.
\end{document}